\documentclass[12pt,a4paper,leqno]{article}

\usepackage[utf8]{inputenc}
\usepackage{amsmath}
\usepackage{amsfonts}
\usepackage{amsthm}
\usepackage{calrsfs}
\usepackage{amssymb}

\usepackage{dsfont}
\usepackage{mathtools}
\usepackage{color}

\usepackage{geometry}
\newtheorem{theorem}{Théorème}

\newtheorem{rmq}{\textbf{Remark} }

\newtheorem{prop}[theorem]{\textbf{Proposition}}

\newenvironment{dem}{\textbf{Proof}}{\qed \\ }

\newtheorem{corol}[theorem]{\textbf{Corollary}}

\newtheorem{lemm}[theorem]{\textbf{Lemma}}
\newtheorem{thm}[theorem]{\textbf{Theorem}}

\newtheorem*{thm*}{\textbf{Theorem}}
\newtheorem{assump}{\textbf{Assumption}}

\newtheorem*{fact}{\textbf{Fact}}

\usepackage{filecontents,lipsum}

\usepackage{mystyle}

\newcommand{\R}{\mathbb{R}}
\newcommand{\N}{\mathbb{N}}

\title{Stability of higher order eigenvalues in dimension one}
\author{Jordan Serres\footnote{Institut de Mathématiques de Toulouse, jordan.serres@math.univ-toulouse.fr}}

\begin{document}

\maketitle

\begin{center}
\textbf{Abstract}
\end{center}
\begin{quotation}
We study stability of the eigenvalues of the generator of a one dimensional reversible diffusion process satisfying some natural conditions. The proof is based on Stein's method. In particular, these results are applied to the Normal distribution (via the Ornstein-Uhlenbeck process), to Gamma distributions (via the Laguerre process) and to Beta distributions (via Jacobi process).
\end{quotation}

\section{Introduction}

A classical question in Spectral Geometry is to identify properties of a manifold from the knowledge of eigenvalues of a canonical differential operator.
The most extensively studied case is when the differential operator is the Laplace-Beltrami operator of a Riemannian manifold.
This problem has been formulated by the famous "Can one hear the shape of a drum?" by M.Kac \cite{drum}. We refer the reader to \cite{nodrum1,nodrum2,Milnordrum} and to the survey \cite{surveydrum}.
The Hille-Yosida theory gives that under certain natural conditions, a differential operator generates a contractive semigroup (see \cite{Yosida}). In particular, in the case of the Laplace-Beltrami operator, it is the heat semigroup. In this range of ideas, Kato's formula implies comparison of semigroups and hence comparison of eigenvalues (see \cite{bessonkatoineq, katogeneralization, SIMON}).
A large part of the literature is also devoted to estimates of the growth of eigenvalues of Schrodinger operators. These include the works of M.Bordoni \cite{bordonischrodingeroperator}, A.Laptev \cite{laptev} and E.Lieb and W.Thirring \cite{liebschoperator, LiebThirring}.

There are many classical comparison results involving only the first or second eigenvalues of operators. Let us cite among them the celebrated Faber-Krahn inequality: balls uniquely minimize the first Dirichlet eigenvalue of the Lapacian in $\R^d$ among sets with given volum \cite{faber1923, krahn1925}, and the Hong-Krahn-Szego inequality: disjoint pair of equal balls uniquely minimize the second Dirichlet eigenvalue among sets with given volum \cite{ krahn1926, hong, polyiaszego}.

In terms of functional inequalities, the first eigenvalue is encoded by the Poincaré constant. A probability measure $\mu$ on $\mathbb{R}^d$ is said to satisfy a Poincaré inequality when for all functions $f$ in the Sobolev space $H^{1}(\mu)$,
\begin{equation}
\mathrm{Var}_\mu (f) \leq C_P(\mu) \int |\nabla f|^2 d\mu,  
\end{equation}
where $C_P(\mu)$ denotes the smallest constant for which the above inequality holds.
Poincaré inequalities have many applications (see for instance the survey \cite{surveyineg}). When $\mu$ is reversible for a Markov process, the infinitesimal generator $L$ of the Markov process is symmetric on $L^2(\mu)$ and the quantity $\lambda_1(\mu) := \frac{1}{C_P(\mu)}$ is then the spectral gap of the positive symmetric operator $-L$ (see \cite[section 4.2.1]{BGL}).

Stability results for Poincaré constant began to appear in the late 80's. Chen \cite[Corollary 2.1]{1987} showed that all isotropic probability measures on $\mathbb{R}^d$ have sharp Poincaré constant greater than $1$. He proved furthermore that the standard Gaussian is the only one attaining $1$. Then Utev \cite{U} refined this result in dimension one, quantifying the difference between Poincaré constants in term of total variation distance: $$C_P(\nu)\geq 1 + \frac{1}{9}d_{TV}(\nu,\gamma)^2 $$ where $\nu$ is a normalized probability measure on $\mathbb{R}$, $\gamma$ is the standard Gaussian and $d_{TV}$ is the total variation distance.
More recently, Courtade, Fathi and Pananjady \cite{existsteinkernel}, extended it to the multidimensional case with the Wasserstein-$2$ distance:
\begin{equation}\label{stabconnu}
C_P(\nu)\geq 1 + \frac{W_2(\nu,\gamma)^2}{d} 
\end{equation}
where $\nu$ is a centered probability measure on $\mathbb{R}^d$, normalized such that $\int|x|^2\,d\nu=d$, $\gamma$ denotes the Gaussian $\mathcal{N}(0,I_d)$ and $W_2$ is the $2$-Wasserstein distance (see \cite[chapter 6]{villani}). 
This result has been extended in a more abstract setting, for a general reference probability measure $\mu$ on a manifold instead of the Gaussian on $\mathbb{R}^d$.
\begin{thm} \cite[Theorem 16]{js}
Let $L$ be a Markov reversible generator with respect to a probability measure $\mu$, carré du champ operator $\Gamma$, and with spectral gap $C_P(\mu)^{-1}$ and associated eigenfunction $f_1$. If any other measure $\nu$ satisfies the normalization conditions $$\int f_1\,d\nu=0,\quad \int f_1^2\,d\nu=1,\quad \int\Gamma(f_1)\,d\nu\leq \frac{1}{C_P(\mu)},$$ and the Poincaré inequality $$\forall f,\quad \mathrm{Var}_\nu(f)\leq C_P(\nu)\int\Gamma(f)\,d\nu, $$ then $C_P(\nu)\geq C_P(\mu)$ and moreover the closeness between $C_P(\nu)$ and $C_P(\mu)$ bounds the $1$-Wasserstein distance between the laws of the pushforwards of $\mu$ and $\nu$ by $f_0$:
\begin{equation}\label{stabilityresultonedimension}
W_1\left(f_0^\#(\mu)\,,\,f_0^\#(\nu)\right)\leq Const\left(\frac{1}{C_P(\mu)}\sqrt{C_P(\nu)-C_P(\mu)}+ \frac{\sqrt{C_P(\nu)}}{C_P(\mu)^2}\left(C_P(\nu)-C_P(\mu)\right) \right).
\end{equation}
The constant is finite when the generalized gradient $\Gamma(f_1)$ of the eigenfunction satisfies some  growth conditions (see Proposition $\ref{gammasteinbound}$).
\end{thm}
In this paper we will consider a diffusion process $L$ on an interval with reversible probability measure $\mu$ and carré du champ operator $\Gamma$. The reason why we obtain results only in dimension one will appear clear in Section $\ref{stabilityresultpart}$. However, the entire framework and results outlined up to section $\ref{stabilityresultpart}$ remain valid in higher dimensions. We will derive a stability result for higher order eigenvalues of the generator $L$.
In \cite{js}, we used the following min-max theorem as definition of the first eigenvalue: 
\begin{equation}\label{minmaxfirsteigenvalue}
\lambda_1(\nu)=\underset{f\in H^1(\nu)}{\inf}\frac{\int \Gamma(f)\,d\nu}{\int f^2\,d\nu}. 
\end{equation}
There are other min-max theorems for higher order eigenvalues, that require to change the functional space over which the infimum in Formula ($\ref{minmaxfirsteigenvalue}$) runs. This can be seen as improving the Poincaré constant by decreasing the domain of the inequality. Despite the changes, we shall see that the main ingredients used for the stability of the first eigenvalue can still by used to establish stability results for higher order eigenvalues.

Let $k\in\N$, $k\geq 1$ and $0<\lambda_1(\mu)<\lambda_2(\mu)<...<\lambda_k(\mu)<... $ be the sequence of eigenvalues of $-L$, counted without multiplicity, and let $f_k$ be a normalized eigenfunction associated with $\lambda_k(\mu)$, i.e. $$ -Lf_k=\lambda_k\,f_k, \quad \int f_k\,d\mu=0,\quad \mathrm{and}\quad \int f_k^2d\mu=1.$$ We set $I_k:=\mathrm{Im}(f_k)$, $a_k:=\inf I_k$ and $b_k:=\sup I_k$.
Let $\nu$ be another probability measure on $M$, normalized so that 
$$
\int f_k\,d\nu=0,\quad \int f_k^2\,d\nu=1,\quad \int\Gamma(f_k)\,d\nu\leq \lambda_k(\mu),
$$
and satisfying the following improved Poincaré inequality
$$
 \int f^2\,d\nu \leq \frac{1}{\lambda_k(\nu)}\int\Gamma(f)\,d\nu\quad \forall f\in H^1(\nu)\cap \left(Sp_1(\nu)\oplus...\oplus Sp_{k-1}(\nu) \right)^{\perp},
$$ where $Sp_i(\nu)$ denotes the $i$-th eigenspace of $\nu$, $\lambda_k(\nu)$ denotes the $k$-th eigenvalue of $\nu$, and the orthogonal complement is to be understood in the $L^2(\nu)$ sense. We will show (see Lemma $\ref{lemmcomparhigh}$) that under these normalization conditions, $\nu$ satisfies 
$$ \lambda_k(\nu)\leq \lambda_k(\mu)+ \sum_{i=1}^{k-1} \left( \lambda_k(\nu)-\lambda_i(\nu)\right)d(f_k,Sp_i(\nu))^\perp)^2,$$ where $d(f_k,Sp_i(\nu))^\perp)^2$ denotes the squared distance between $f_k$ and $Sp_i(\nu)^\perp$. We refine this by proving the following stability result for the $k$-th eigenvalue.
\begin{thm}\label{thmintrohigh}
For all one dimensional probability measures $\nu$ normalized as in $(\ref{normalisation})$, satisfying the improved Poincaré inequalities $(\ref{poicareimproved})$, and the technical Assumption $\ref{hjcommeilfaut}$, it holds for some finite constant $C>0$:

$$ \sum_j \nu_j(I_k^j)\, W_1(\nu_j^*,\mu_j^*) \leq C \left[\sqrt{\left|\lambda_k(\mu)-\lambda_k(\nu)\right|} +\frac{ \left|\lambda_k(\mu)-\lambda_k(\nu)\right|}{\sqrt{\lambda_1(\nu)}} + \sum_{i=1}^{k-1}C_i\, d(f_k,Sp_i(\nu)^\perp)\right]  $$
where $(I_k^j)_j$ are the images by $f_k$ of the connected components of the complementary of its critical points, $\nu_j^*$ (resp. $\mu_j^*$) is the pushforward of $\nu$ (resp. $\mu$) restricted to $I_k^j$, constants $C_i$ are given by $$C_i=\sqrt{\lambda_k(\nu)-\lambda_i(\nu)}+\frac{\lambda_k(\nu)-\lambda_i(\nu)}{\sqrt{\lambda_i(\nu)}},$$ and $d(f_k,Sp_i(\nu)^\perp)$ is defined in Remark $\ref{orthoerror}$. The value $C=\sum_j C_{h_j}^2$ suffices, with $C_{h_j}$ given in Proposition $\ref{gammasteinbound}$.
\end{thm}

The technical Assumption $\ref{hjcommeilfaut}$ (see Section $\ref{steindecoupe}$) ensures the finiteness of the constant $C$ (see Proposition $\ref{resumechfini}$), and asks that the carré du champ operator $\Gamma(f_k)$ of the eigenfunction satisfies a certain polynomial growth condition. The proof of Theorem $\ref{thmintrohigh}$ is based on an approximate integration by parts formula satisfied by $\nu$ with respect to the $k$-th eigenfunction $f_k$ (see Corollary $\ref{meilleureeqippapproxavecf2}$) and the use of Stein's method to the pushforward of $\mu$ by $f_k$ (see Section $\ref{implementsteinhigh}$). Let us mention that the exact integration by parts satisfied by the $k$-th Hermite polynomial in case of the Normal distribution, was used in \cite{higherordersteinkernel} to define the notion of higher order Stein's kernels in the context of Gaussian approximation.

Our result applies in particular to the normal distribution (see Section $\ref{highergauss}$), Gamma distributions on $\mathbb{R}_+$ (see Section $\ref{higherlaguerre}$), and Beta distributions on $[-1,1]$ (see Section $\ref{highjacobi}$). Applying this to the second Hermite polynomial, we obtain the following Chi-2 approximation result.
For all measure $\nu$ on $\mathbb{R}$ normalized as $$\int x^2d\nu=1,\quad \mathrm{and}\quad \int x^4\,d\nu = 3, $$ it holds for some finite positive constant $C>0$:
$$W_1\left(\frac{1}{\sqrt{2}}\left(\chi_2 -1\right),\nu^*\right) \leq C \left[\sqrt{\left|2-\lambda_2(\nu)\right|} +\frac{\left|2-\lambda_2(\nu)\right|}{\sqrt{\lambda_1(\nu)}} + C_\nu\, d\left(\frac{1}{\sqrt{2}}\left(x^2-1\right),Sp_1(\nu)^\perp\right) \right]$$
where $\chi_2$ is the $\chi_2$-distribution on $\mathbb{R}_+$, $\nu^*$ is the pushforward of $\nu$ by the second Hermite polynomial $\frac{1}{\sqrt{2}}\left(x^2-1\right)$, the constant $C_\nu$ is given by $C_\nu=\sqrt{\lambda_2(\nu)-\lambda_1(\nu)} +\frac{\lambda_2(\nu)-\lambda_1(\nu)}{\sqrt{\lambda_1(\nu)}}$, and $d\left(\frac{1}{\sqrt{2}}\left(x^2-1\right),Sp_1(\nu)^\perp\right)$ quantifies the orthogonality error between $\frac{1}{\sqrt{2}}\left(x^2-1\right)$ and the first eigenspace of $\nu$ (see Section $\ref{sectioneigenspacenu}$).\\

Let us say a few words about stability results from a geometric setting. The Lichnerowicz theorem asserts that among all Riemannian manifolds with Ricci curvature bounded by below by $N-1$, unit spheres of dimension $N$ uniquely minimizes the first eigenvalue of the Laplace-Beltrami operator \cite{Lichnerowicz}. The Bakry-Emery criterion \cite{bakryEmery, bakry1989} extends this result to Gaussian spaces: if $\mu=e^{-V}dx$ is a probability distribution which is more log-concave than the Gaussian (i.e. $\mathrm{Hess}\,V \geq I_d $) then its Poincaré constant is smaller than $1$ which is that of the Gaussian.

While the original proof of Bakry-Emery is based on the semigroup method, another powerful method is the contraction principle \cite{Milman2018}: if $\mu$ is the pushforward of $\nu$ by a $L$-Lipshitz map, then $C_P(\mu)\leq L \,C_P(\nu)$. In particular, Caffarelli's contraction theorem \cite{caffarelli1} states that the optimal transport between the Gaussian and a more log-concave distribution given by the Brenier map is $1$-Lipschitz, recovering the Bakry-Emery criterion. E.Milman \cite{Milman2018} pointed out that the contraction principle does not only entail a comparison between the first eigenvalues, but a comparison between the entire spectra. In that area of sprectral comparison by the contraction principle, let us cite the recent works of D.Mikulincer and Y.Shenfeld \cite{mikulincershefeld1, mikulincershefeld2}.

The question of stability of spectral estimates has been addressed in various works. We refer the reader to the survey \cite{brasco2016spectral} by L.Brasco and G. De Philippis for a view of quantitative sharp inequalities for first (and second) eigenvalues of the Laplacian in $\R^d$. Let us mention in particular the quantitative form of Faber-Krahn inequality. L.Brasco, G. De Philippis and B.Velichkov proved in \cite{MR3357184} that there exists a constant $\sigma>0$ depending on the dimension such that for all $\Omega\subset\mathbb{R}^d$ of volum $1$, 
\begin{equation*}
\lambda_1(\Omega)\geq \lambda_1(B)+ \sigma \mathcal{A}(\Omega)^2
\end{equation*}
where $\lambda_1(\Omega)$ denotes the first Dirichlet eigenvalue of the Laplacian on $\Omega$, $B$ denotes the unit ball in $\mathbb{R}^d$, $\mathcal{A}(\Omega)$ is the Fraenkel asymetry of $\Omega$, and the exponent $2$ is sharp.

The study of the stability of the spectral gap of a diffusion operator falls within this framework. Under the curvature-dimension condition, let us cite the work \cite{BF21} of J.Bertrand and M.Fathi which treats the case of the positive curvature and the infinite dimension. In particular, they show that any $RCD(1,\infty)$ space reaching almost the Bakry-Emery bound $1$ for its spectral gap, admits approximately all integers in its spectrum. The stability is quantified in terms of a spectral comparison with the Gaussian, since the integers are eigenvalues of the Ornstein-Uhlenbeck generator for which the Gaussian is reversible. In case of the positive curvature and the finite dimension, the Lichnerowicz theorem has been extended in the following way.
\begin{thm} \cite[Theorem 1.1]{rcdstability}
Let $(M, d , \mu)$ be an RCD$(N-1,N)$ space with $N > 1$ and spectral gap $\lambda_1 \leq N + \varepsilon$ for some $\varepsilon > 0$, with $f$ an eigenfunction of the Laplacian, with eigenvalue $\lambda_1$ and normalized so that $||\Gamma(f)||_1 = N/(N+1)$. There is a constant $C(N) > 0$ (independent of $M$) such that the $1$-Wasserstein distance between the pushforward of $\mu$ by $f$ and a symmetrized Beta distribution with parameters $(N/2, N/2)$ is smaller than $C(N)\varepsilon$. 
\end{thm}
Let us emphasize that the stability is quantified in terms of $W_1$ distance between pushforward by the first eigenfunction, since the symmetrized Beta distribution with parameters $(N/2, N/2)$ is the distribution of the pushforward by a first eigenfunction of the reversible law of the Laplacian on a sphere. 
Let us conclude this introduction by mentioning that under the normalisation approach used in this paper, the spectral gap of the model space is maximal, whereas under the curvature condition approach, it is minimal.

\section{The space of normalized probability distributions $\nu$}
In this section, we explicitly describe the space of normalized probability distributions on which our stability result holds. We consider a probability measure $\nu$ such that 
\begin{equation}\label{normalisation}
\int f_k\,d\nu=0,\quad \int f_k^2\,d\nu=1,\quad\mathrm{and}\quad \int\Gamma(f_k)\,d\nu\leq \lambda_k(\mu) 
\end{equation} 
Let us underline that these normalization conditions correspond to 
$$\int f_k\,d\nu=\int f_k\,d\mu,\quad \int f_k^2\,d\nu=\int f_k^2\,d\mu,\quad \int\Gamma(f_k)\,d\nu\leq \int\Gamma(f_k)\,d\mu. $$ But since $f_k$ is an eigenfunction associated to the $k$-th eigenvalue of $\mu$, we have $\int f_k\,d\mu=0$ and $\int\Gamma(f_k)\,d\mu = \lambda_k(\mu)\int f_k^2d\mu$, hence we normalize $f_k$ by $\int f_k^2=1$ in order to make the conditions more readable.
 
\subsection{Eigenspaces of $\nu$}\label{sectioneigenspacenu}
We define the eigenspaces of $\nu$ in the following way: first  $$Sp_1(\nu):=\left\{f\in H^1(\nu)\,|\,\forall g\in H^1(\nu),\, \int fg\,d\nu = \frac{1}{\lambda_1(\nu)}\int\Gamma(f,g)\,d\nu\right\},$$ 
where $$H^1(\nu):=\left\{f\in L^2(\nu)\,|\,\int f\,d\nu=0,\,\int\Gamma(f)\,d\nu<\infty\right\}, $$ and 
\begin{equation}\label{minmaxfirst}
\lambda_1(\nu)=\underset{f\in H^1(\nu)\setminus \{0\}}{\inf}\frac{\int \Gamma(f)\,d\nu}{\int f^2\,d\nu}. 
\end{equation} 
This definition corresponds to eigenspace in a weak sense.
It is clearly a linear space and a subset of $\{f\in H^1(\nu)\,|\,\int f^2\,d\nu = \frac{1}{\lambda_1(\nu)}\int\Gamma(f)\,d\nu \}$. Moreover, if $\nu$ is reversible for some generator $L_\nu$ with carré du champ operator $\Gamma$, then the converse set inclusion holds and $Sp_1(\nu)$ is an eigenspace of $L_\nu$ in the classical sense.
We can then recursively define higher order eigenspaces in a similar way.
\begin{align*}
Sp_{k+1}(\nu):= & \left\{ f\in H^1(\nu)|\,\forall g\in H^1(\nu)\cap \left(Sp_1(\nu)\oplus...\oplus Sp_k(\nu) \right)^{\perp}, \int fg\,d\nu = \frac{1}{\lambda_{k+1}(\nu)}\int\Gamma(f,g)\,d\nu \right\}\\
& \cap \left(Sp_1(\nu)\oplus...\oplus Sp_k(\nu) \right)^{\perp}\,,
\end{align*} where the orthogonal complement is to be understood in the $L^2(\nu)$ sense and $$\lambda_{k+1}(\nu):=\underset{f\neq 0}{\underset{f\in H^1(\nu)\cap \left(Sp_1(\nu)\oplus...\oplus Sp_{k}(\nu) \right)^{\perp} }{\inf}}\frac{\int \Gamma(f)\,d\nu}{\int f^2\,d\nu}. $$
Note that by construction, eigenspaces are pairwise orthogonal in $L^2(\nu)$, and eigenvalues are ordered: $\lambda_1(\nu)\leq \lambda_2(\nu)\leq \cdot\cdot\cdot\leq \lambda_k(\nu)\leq \cdot\cdot\cdot $

Let us emphasize that the integration by parts formula $$\int fg\,d\nu = \frac{1}{\lambda_{k}(\nu)}\int\Gamma(f,g)\,d\nu, $$ when $f$ is an eigenfunction can be interpreted as an "isometry along $f$ in $Sp_k(\nu)$" between the $L^2(\nu)$-norm and the $H^1(\nu)$-norm. This property is the keystone of Lemma $\ref{lemmcomparhigh}$ and Theorem $\ref{ippapprox}$. 
\subsection{Improved Poincaré inequalities}
By definition of eigenvalues and associated eigenspaces, the probability measure $\nu$ always satifies the following improved Poincaré inequalities. 
\begin{equation}\label{poicareimproved}
 \int f^2\,d\nu \leq \frac{1}{\lambda_k(\nu)}\int\Gamma(f)\,d\nu\quad \forall f\in H^1(\nu)\cap \left(Sp_1(\nu)\oplus...\oplus Sp_{k-1}(\nu) \right)^{\perp} 
 \end{equation}  where $$\lambda_k(\nu):=\underset{f\in H^1(\nu)\cap \left(Sp_1(\nu)\oplus...\oplus Sp_{k-1}(\nu) \right)^{\perp} }{\inf}\frac{\int \Gamma(f)\,d\nu}{\int f^2\,d\nu}\geq\lambda_{k-1}(\nu). $$ 
Even if the eigenvalue $\lambda_k(\nu)$ is trivial (i.e. is zero), the improved Poincaré inequality becomes itself trivial, but remains true. 
 
\subsection{Projection of the eigenfunction $f_k$ onto eigenspaces of $\nu$}

The first idea used in the previous study on the spectral gap \cite{js} was to evaluate the Poincaré inequality satisfied by $\nu$ with the first eigenfunction of $\mu$. We want to do the same in the general case, however it is impossible to evaluate the improved Poincaré inequality ($\ref{poicareimproved}$) with $f_k$ since we have no guarantee that $f_k\in H^1(\nu)\cap \left(Sp_1(\nu)\oplus...\oplus Sp_{k-1}(\nu) \right)^{\perp} $. But this space is a linear subspace of $L^2(\nu)$, hence it seems natural to think that ($\ref{poicareimproved}$) should not be evaluated with $f_k$, but with the $L^2$-projection of $f_k$ on $\left(Sp_1(\nu)\oplus...\oplus Sp_{k-1}(\nu) \right)^{\perp} $. 

Let $p_k^\perp$ be the $L^2(\nu)$ orthogonal projection of $f_k$ onto $\left(Sp_1(\nu)\oplus...\oplus Sp_{k-1}(\nu) \right)^{\perp}$ and $p_k$ the $L^2(\nu)$ orthogonal projection of $f_k$ onto $Sp_1(\nu)\oplus...\oplus Sp_{k-1}(\nu) $. Hence we have the following formulas that we will repeatedly use in the sequel:

\begin{align}
f_k & \label{projection1} =p_k+p_k^\perp,\quad p_k\in Sp_1(\nu)\oplus...\oplus Sp_{k-1}(\nu),\quad p_k^\perp\in \left(Sp_1(\nu)\oplus...\oplus Sp_{k-1}(\nu) \right)^{\perp}, \\
p_k & \label{projection2} = p_k^1+...+p_k^{k-1}, \quad p_k^1\in Sp_1(\nu),\, ...\, ,p_k^{k-1}\in Sp_{k-1}(\nu)
\end{align}

Let us point out that in the case of the spectral gap (i.e. $k=1$), $p_1$ would correspond to the projection of $f_1$ onto the kernel of $L$, which is the set of constant functions, and $p_1$ would hence be the projection of $f_1$ onto the set of centered functions. But since $f_1$ is centered, we would have $p_1=f_1$ and so this coincides with the general case where we will use $p_k^\perp$ to evaluate in ($\ref{poicareimproved}$).

\subsection{Eigenvalue comparisons}

In this section, we will show that any probability distribution $\nu$ normalized as ($\ref{normalisation}$) has its $k$-th eigenvalue $\lambda_k(\nu)$ controlled by $\lambda_k(\mu)$, some terms quantifying the distance between $f_k$ and $\left(Sp_1(\nu)\oplus...\oplus Sp_{k-1}(\nu) \right)^{\perp} $, and the gap between succesive eigenvalues of $\nu$. This estimate holds without any additional assumption on $\mu$.
\begin{lemm}\label{lemmcomparhigh}
Let $\nu$ be a probability distribution normalized as in $\eqref{normalisation}$. Then
\begin{equation}\label{debudestabilite}
\lambda_k(\nu)\leq \lambda_k(\mu) +\sum_{i=1}^{k-1} \left( \lambda_k(\nu)-\lambda_i(\nu)\right)\int (p_k^i)^2\,d\nu,
\end{equation}
where $p_k^i$ are the projections defined in $\eqref{projection2}$.
\end{lemm}
\begin{proof}
The proof only consists in evaluating (\ref{poicareimproved}) with $f=p_k^\perp$, which actually belongs to the correct space. On the one hand, using (\ref{projection2}) $$
\int(f_k-p_k)^2d\nu = \int f_k^2\,d\nu+\sum_{i=1}^{k-1}\int (p_k^i)^2\,d\nu -2\int f_k p_k\,d\nu = 1 - \sum_{i=1}^{k-1}\int (p_k^i)^2\,d\nu .$$ On the other hand, using that all $p_k^i$ are eigenfunctions, Formula $\eqref{projection1}$ and Formula $\eqref{projection2}$,
\begin{align*}
\int\Gamma(f_k-p_k)\,d\nu & = \int\Gamma(f_k)\,d\nu + \sum_{i=1}^{k-1}\int\Gamma(p_k^i)\,d\nu -2\int\Gamma(f_k,p_k)\,d\nu \\
 & = \lambda_k(\mu)- \sum_{i=1}^{k-1}\int\Gamma(p_k^i)\,d\nu \\
 & = \lambda_k(\mu)- \sum_{i=1}^{k-1} \lambda_i(\nu)\int (p_k^i)^2\,d\nu.
\end{align*}  We then apply $\eqref{poicareimproved}$ to $p_k^\perp = f_k-p_k\in \left(Sp_1(\nu)\oplus...\oplus Sp_{k-1}(\nu) \right)^{\perp}$ and get $$\lambda_k(\nu) -\lambda_k(\nu)\sum_{i=1}^{k-1}\int (p_k^i)^2\,d\nu\leq \lambda_k(\mu)-\sum_{i=1}^{k-1} \lambda_i(\nu)\int (p_k^i)^2\,d\nu$$ which gives the result.
\end{proof}
Let us point out that equality holds in $(\ref{debudestabilite})$ if $\nu=\mu$ because in that case, 
$f_k$ belongs to $\left(Sp_1(\nu)\oplus...\oplus Sp_{k-1}(\nu) \right)^{\perp}$ and thus $p_k=0$.
The first natural question is then about rigidity of inequality ($\ref{debudestabilite}$). What can we say about $\nu$ if the inequality ($\ref{debudestabilite}$) is in fact an equality? We have seen in case of the spectral gap ($k=1$) that this implies the pushforward measures $f_1^{\#}\mu$ and $f_1^{\#}\nu$ to be equal, but the measures $\mu$ and $\nu$ themselves can be different. We will see that for general $k\geq 2$, the equality case also implies some link between the pushforward $f_k^{\#}\mu$ and $f_k^{\#}\nu$, which itself implies equality of the pushforward measures in case where $\nu$ allocate the same weight as $\mu$ on each non critical sets of $f_k$ (see Section $\ref{splitpushforward}$).

\begin{rmq}\label{orthoerror}
The quantity $\int (p_k^i)^2\,d\nu$ is the square distance between $f_k$ and $Sp_i(\nu)^\perp$ (by definition of the projection), and quantifies therefore the orthogonality error between $f_k$ and eigenspaces of lower orders of $\nu$. We denote it by $d(f_k,Sp_i(\nu))^\perp)^2$.
Therefore $(\ref{debudestabilite})$ becomes: 
\begin{equation}\label{debutdestabiliteprecise}
\lambda_k(\nu)\leq \lambda_k(\mu)+ \sum_{i=1}^{k-1} \left( \lambda_k(\nu)-\lambda_i(\nu)\right)d(f_k,Sp_i(\nu))^\perp)^2
\end{equation}
\end{rmq}

\section{Approximate Integration by Part formula}

In this section, we derive approximate integration by parts formulas, for the measure $\nu$, with an error term involving quantities appearing in the comparison ($\ref{debutdestabiliteprecise}$) between the eigenvalues of $\mu$ and $\nu$. This approximate integration by parts formula will be the keystone to use Stein's method in this context. We shall proceed as for the spectral gap, with a difference: we now use $p_k^\perp$ instead of directly using $f_k$ as minimizer in the improved Poincaré inequality. Hence, in a first step, we will derive an approximate integration by parts formula  with $p_k^\perp$ and only valid on $\left(Sp_1(\nu)\oplus...\oplus Sp_{k-1}(\nu) \right)^{\perp}$. In a second step, we will replace $p_k^\perp$ by $f_k$, and finally in a third step we will extend it to the whole space $H^1(\nu)$.

\begin{thm}\label{ippapprox}
We have the following inequality for all $g\in H^1(\nu)\cap \left(Sp_1(\nu)\oplus...\oplus Sp_{k-1}(\nu) \right)^{\perp}$,
\begin{equation}\label{eqippapprox}
\left|\int \left( \lambda_k(\nu)p_k^\perp g - \Gamma(p_k^\perp,g)\right)\,d\nu \right| \leq \left[ \lambda_k(\mu)-\lambda_k(\nu)+ \sum_{i=1}^{k-1} \left( \lambda_k(\nu)-\lambda_i(\nu)\right)\,d(f_k,Sp_i(\nu))^\perp)^2\right]^\frac{1}{2}\sqrt{\int\Gamma(g)\,d\nu} 
\end{equation} where $p_k^\perp$ is defined in $(\ref{projection1})$.
\end{thm}
\begin{dem}
Let $t\in\mathbb{R}$ and $g\in H^1(\nu)\cap \left(Sp_1(\nu)\oplus...\oplus Sp_{k-1}(\nu) \right)^{\perp}$. Let us apply $(\ref{poicareimproved})$ to $\alpha:= p_k^\perp + tg \in H^1(\nu)\cap \left(Sp_1(\nu)\oplus...\oplus Sp_{k-1}(\nu) \right)^{\perp}$. Computing
\begin{align*}
 \mathrm{Var}_{\nu}(\alpha) & = \int (p_k^\perp+tg)^2 \,d\nu = \int (p_k^\perp)^2 d\nu + 2t\int p_k^\perp g\,d\nu + t^2\int g^2d\nu \\ 
 & = \int(f_k-p_k)^2d\nu + 2t\int p_k^\perp g\,d\nu + t^2\int g^2d\nu \\
 & = 1- \int p_k^2\,d\nu + 2t\int p_k^\perp g\,d\nu + t^2\int g^2d\nu \\
 & = 1- \sum_{i=1}^{k-1}\int (p_k^i)^2\,d\nu + 2t\int p_k^\perp g\,d\nu + t^2\int g^2d\nu \\
 & = 1- \sum_{i=1}^{k-1}\frac{1}{\lambda_i(\nu)}\int \Gamma(p_k^i)\,d\nu + 2t\int p_k^\perp g\,d\nu + t^2\int g^2d\nu,
\end{align*}
and
\begin{align*}
 \int\Gamma(\alpha)d\nu & = \int \Gamma(p_k^\perp)d\nu  + 2t\int\Gamma(p_k^\perp,g)d\nu +t^2\int\Gamma(g)d\nu \\
 & = \int \Gamma(f_k-p_k)\,d\nu + 2t\int\Gamma(p_k^\perp,g)d\nu +t^2\int\Gamma(g)d\nu \\
 & = \lambda_k(\mu)+\int\Gamma(p_k)\,d\nu -2\int\Gamma(f_k,p_k)\,d\nu+2t\int\Gamma(p_k^\perp,g)d\nu +t^2\int\Gamma(g)d\nu \\
 & = \lambda_k(\mu)-\int\Gamma(p_k)\,d\nu +2t\int\Gamma(p_k^\perp,g)d\nu +t^2\int\Gamma(g)d\nu \\
  & = \lambda_k(\mu)-\sum_{i=1}^{k-1}\int\Gamma(p_k^i)\,d\nu +2t\int\Gamma(p_k^\perp,g)d\nu +t^2\int\Gamma(g)d\nu,
\end{align*}
where we have used at line $4$ that, since all $p_k^i$ are eigenfunctions, 
\begin{align*}
\int\Gamma(f_k,p_k)\,d\nu & = \sum_i \int\Gamma(p_k^\perp,p_k^i)\,d\nu+ \sum_{i,j}\int\Gamma(p_k^i,p_k^j)\,d\nu\\
& = \sum_i \lambda_i(\nu)\int p_k^\perp p_k^i\,d\nu + \sum_{i,j} \lambda_i(\nu)\int p_k^ip_k^j\,d\nu\\
& = \sum_i\lambda_i(\nu)\int (p_k^i)^2d\nu\\
& = \sum_i \int \Gamma(p_k^i)\,d\nu\\
&= \int \Gamma(p_k)\,d\nu,
\end{align*}
we get that for all $t\in\mathbb{R}$, the degree two polynomial
$$
  \left(-\int \Gamma(g)\,d\nu\right) t^2 + 2\left( \int \left( \lambda_k(\nu)p_k^\perp g - \Gamma(p_k^\perp,g)\right)d\nu \right)t + \left( \lambda_k(\nu)-\lambda_k(\mu)+\sum_{i=1}^{k-1}\left(1-\frac{\lambda_k(\nu)}{\lambda_i(\nu)}\right)\int\Gamma(p_k^i)\,d\nu\right)
  $$ is non positive. 
Hence its discriminant is non positive: $$4\left( \int \left( \lambda_k(\nu)p_k^\perp g - \Gamma(p_k^\perp,g)\right)d\nu \right)^2 +4\int\Gamma(g)\,d\nu\left( \lambda_k(\nu)-\lambda_k(\mu)+\sum_{i=1}^{k-1}\left(1-\frac{\lambda_k(\nu)}{\lambda_i(\nu)}\right)\int\Gamma(p_k^i)\,d\nu\right) \leq 0 $$ which gives the result since $\int\Gamma(p_i^k)\,d\nu=\lambda_i(\nu)\int (p_k^i)^2d\nu = \lambda_i(\nu)\,d(f_k,Sp_i(\nu))^\perp)^2 $.
\end{dem}

Now it is easy to see that one can replace $p_k^\perp$ by $f_k$ without any additional cost.
\begin{corol}
We have the following inequality for all $g\in H^1(\nu)\cap \left(Sp_1(\nu)\oplus...\oplus Sp_{k-1}(\nu) \right)^{\perp}$,
\begin{equation}\label{eqippapproxavecf2}
\left|\int \left( \lambda_k(\nu)f_k\, g - \Gamma(f_k,g)\right)\,d\nu \right| \leq \left[ \lambda_k(\mu)-\lambda_k(\nu)+ \sum_{i=1}^{k-1} \left( \lambda_k(\nu)-\lambda_i(\nu)\right)d(f_k,Sp_i(\nu))^\perp)^2\right]^\frac{1}{2}\sqrt{\int\Gamma(g)\,d\nu} 
\end{equation}
\end{corol}
\begin{dem}
Use the fact that $p_k^\perp = f_k -p_k$ in Theorem $\ref{ippapprox}$ and both $\int p_k\, g\,d\nu = 0$ and $\int \Gamma(p_k,g)\,d\nu=0 $ since $p_k\in Sp_1(\nu)\oplus...\oplus Sp_{k-1}(\nu)$ and $g\in H^1(\nu)\cap \left(Sp_1(\nu)\oplus...\oplus Sp_{k-1}(\nu) \right)^{\perp}$.
\end{dem}

Finally, one can extend the approximate integration by parts ($\ref{eqippapproxavecf2}$) on the whole $H^1(\nu)$ and it only adds a term which is again controled by the orthogonal error of the eigenfunction.
\begin{corol}\label{meilleureeqippapproxavecf2}

We have the following inequality for all $g\in H^1(\nu)$:
\begin{equation}\label{ippapproch}
\left|\int \left( \lambda_k(\nu)f_k\, g - \Gamma(f_k,g)\right)\,d\nu \right| \leq \left[\sqrt{\left|\lambda_k(\mu)-\lambda_k(\nu)\right|} + \sum_{i=1}^{k-1}C_i\, d(f_k,Sp_i(\nu)^\perp) \right]\sqrt{\int\Gamma(g)\,d\nu}
\end{equation}
where 
\begin{equation}\label{cnu}
C_i=\sqrt{\left|\lambda_k(\nu)-\lambda_i(\nu)\right|}+\frac{\lambda_k(\nu)-\lambda_i(\nu)}{\sqrt{\lambda_i(\nu)}}
\end{equation}

\end{corol}

\begin{rmq}
We cannot avoid the absolute value under the square root because we only know that $\eqref{debutdestabiliteprecise}$ holds, which does not imply $\lambda_k(\nu)\leq \lambda_k(\mu)$ except for $k=1$.
\end{rmq}

\begin{dem}
Let $g\in H^1(\nu)$. Let $g=g_P+g_\perp$ with $g_p\in H^1(\nu)\cap \left(Sp_1(\nu)\oplus...\oplus Sp_{k-1}(\nu) \right)$ and $g_\perp \in H^1(\nu)\cap \left(Sp_1(\nu)\oplus...\oplus Sp_{k-1}(\nu) \right)^{\perp}$. We have: $$\int \left( \lambda_k(\nu)f_k\, g - \Gamma(f_k,g)\right)\,d\nu = \int \left( \lambda_k(\nu)f_k\, g_p - \Gamma(f_2,g_p)\right)\,d\nu+\int \left( \lambda_k(\nu)f_2\, g_\perp - \Gamma(f_k,g_\perp)\right)\,d\nu $$ We apply ($\ref{eqippapproxavecf2}$) to the second term in the sum. For the first one, since $g_p=\sum_{i=1}^{k-1} g_p^i \in Sp_1(\nu)\oplus...\oplus Sp_{k-1}(\nu)$, we have
\begin{align*}
\int \left( \lambda_k(\nu)f_k\, g_p - \Gamma(f_k,g_p)\right)\,d\nu = & \sum_{i=1}^{k-1}\left(\lambda_k(\nu)-\lambda_i(\nu)\right)\int f_k\, g_p^i\,d\nu\\
=& \sum_{i=1}^{k-1}\left(\lambda_k(\nu)-\lambda_i(\nu)\right)\int p_k\, g_p^i\,d\nu\,\, \mathrm{using\,\, (\ref{projection1})}\\
=& \sum_{i=1}^{k-1}\left(\lambda_k(\nu)-\lambda_i(\nu)\right)\int p_k^i\, g_p^i\,d\nu\,\, \mathrm{using\,\, (\ref{projection2})}\\
\leq & \sum_{i=1}^{k-1}\left(\lambda_k(\nu)-\lambda_i(\nu)\right)\sqrt{\int (p_k^i)^2\,d\nu}\sqrt{\int (g_p^i)^2\,d\nu}\\
= & \sum_{i=1}^{k-1}\left(\lambda_k(\nu)-\lambda_i(\nu)\right)d(f_k,Sp_i(\nu)^\perp)\sqrt{\frac{1}{\lambda_i(\nu)}\int \Gamma(g_p^i)\,d\nu}\\
\leq & \sum_{i=1}^{k-1}\frac{\lambda_k(\nu)-\lambda_i(\nu)}{\sqrt{\lambda_i(\nu)}}d(f_k,Sp_i(\nu)^\perp)\sqrt{\int\Gamma(g)\,d\nu}.\\
\end{align*}
which allows to conclude.
\end{dem}

\section{Stability result in dimension one}\label{stabilityresultpart}

From now on, our results will only apply when $L$ is a one-dimensional diffusion operator on a (possibly infinite) interval. We will prove the following stability result for the $k$-th eigenvalue of $-L$.

\begin{thm}\label{stabund}
Let $L$ be a diffusion generator on an interval $M\subset\mathbb{R}$, let $0<\lambda_1(\mu)<\lambda_2(\mu)<...<\lambda_k(\mu) $ be its $k\geq 1$ first eigenvalues, counted without multiplicity, and let $f_k$ be an eigenfunction associated with $\lambda_k(\mu)$, satisfying Assumption $\ref{hjcommeilfaut}$. Let $\mathrm{Crit}(f_k)$ be the set of all critical points of $f_k$.
Then for all probability measures $\nu$ on $M$ normalized as in $(\ref{normalisation})$ and satisfying the improved Poincaré inequalities $(\ref{poicareimproved})$, it holds for some finite constant $C>0$:

$$ \sum_j \nu_j(I_k^j)\, W_1(\nu_j^*,\mu_j^*) \leq C \left[\sqrt{\left|\lambda_k(\mu)-\lambda_k(\nu)\right|} +\frac{ \left|\lambda_k(\mu)-\lambda_k(\nu)\right|}{\sqrt{\lambda_1(\nu)}} + \sum_{i=1}^{k-1}C_i\, d(f_k,Sp_i(\nu)^\perp)\right]  $$
where $(I_k^j)_j$ are the images by $f_k$ of the connected components of $M\setminus\mathrm{Crit}(f_k)$, $\nu_j^*$ (resp. $\mu_j^*$) is the pushforward of $\nu$ (resp. $\mu$) restricted to $I_k^j$, constants $C_i$ are given by $$C_i=\sqrt{\lambda_k(\nu)-\lambda_i(\nu)}+\frac{\lambda_k(\nu)-\lambda_i(\nu)}{\sqrt{\lambda_i(\nu)}},$$ and $d(f_k,Sp_i(\nu)^\perp)$ is defined in Remark $\ref{orthoerror}$. The value $C=\sum_j C_{h_j}^2$ suffices, with $C_{h_j}$ given in Proposition $\ref{gammasteinbound}$.
\end{thm}

Let us point out that Theorem $\ref{stabund}$ implies that if $\lambda_k(\nu)=\lambda_k(\mu)$ and if $f_k$ is orthogonal in $L^2(\nu)$ to all lower order eigenspaces of $\nu$, then the conditional pushforward of $\nu$ and $\mu$ are all equal : $\forall j$, $\nu_j^*=\mu_j^*$. In this case, one can compute that for all bounded $\phi:I_k\rightarrow\mathbb{R}$,
\begin{align*}
\int \phi(f_k)\,d\nu-\int\phi(f_k)\,d\mu = & \sum_j\left(\int_{J_j} \phi(f_k)\,d\nu - \int_{J_j}\phi(f_k)\,d\mu \right)\\
= & \sum_j\left( \int_{I_k^j}\nu(J_j)\,\phi\,d\nu_i^* -\int_{I_k^j}\mu(J_j)\,\phi\,d\mu_j^*\right)\\
= & \sum_j \left( \nu(J_j)-\mu(J_j) \right)\int_{I_k^j}\phi\,d\mu_j^*.
\end{align*}
We then deduce the following corollary.
\begin{corol}
Let $\mu$ and $\nu$ such as required in Theorem $\ref{stabund}$. If moreover
\begin{itemize}
\item $\lambda_k(\nu)=\lambda_k(\mu)$,
\item $\forall i\leq k$, $f_k \perp Sp_i(\nu)$ in $L^2(\nu)$, and
\item $\forall j$, $\nu(J_j)=\mu(J_j)$,
\end{itemize}
then the pushforwards by $f_k$ are the same, that is $$ f_k^{\#}\nu = f_k^{\#}\mu. $$
\end{corol}

When $k=1$ the last two conditions are trivially satisfied, so we recover the result in \cite{js}.

\subsection{Taking the pushforward by $f_k$}\label{splitpushforward}

In \cite{js}, after obtaining the approximate integration by parts formula, we pushforwarded it by the first eigenfunction $f_1$. The integration by parts formula then became a one dimensional ODE that we explicitely solved. But taking the pushforward was possible because we assumed that the carré du champ of the first eigenfunction $\Gamma(f_1)$ could be factorize as $\Gamma(f_1)=h\circ f_1$ for some non-negative function $h:I_1\rightarrow \mathbb{R}_+$. We justified this assumption by the fact that in case of dimension one, where $M\subset\mathbb{R}$ is an interval, all first eigenfunctions are known to be strictly monotone, and hence injective, so one can simply take $h:=\Gamma(f_1)\circ f_1^{-1}$. In a multidimensional space $M$, $f_1$ cannot be injective. However, we have the classical example of eigenfunctions on Spheres which also satisfy this factorization assumption. Hence this assumption does not seem so odd.

But now considering higher order eigenfunctions, in dimension one we no longer have monotonicity, and no injectivity either, so the obvious choice $h:=\Gamma(f_k)\circ f_k^{-1}$ is no longer available. Let us nevertheless point out that in our classical examples (i.e. the Normal, Gamma and Beta distributions), the second eigenfunctions of the Normal and Beta distributions still satisfy this factorization assumption, despite being non-injective. This is due to the symetry of the Gaussian and Beta distributions. There is no other such eigenfunctions of any order in these three examples which factorize. However, between critical points, the derivative of any eigenfunction is obviously either positive or negative, and hence we have a local injectivity between critical points, so a local factorization of the carré du champ. Our approach is then to locally take the pushforward of ($\ref{ippapproch}$) on each connected component of $M\setminus\mathcal{C}_k$ where we denote the set of critical points of $f_k$ by $\mathcal{C}_k$.

Let $$\mathcal{C}_k:=\{ x\in M\subset\mathbb{R}\,|\, \Gamma(f_k)(x)=0\}= \mathrm{Crit}(f_k)$$ be  the set of critical points of $f_k$ and assume this set to be finite. This is always the case in classical examples, where $\mathcal{C}_k$ has $k-1$ elements. Let then $(J_j)_j$ be the connected components of $M\setminus\mathcal{C}_k$. Hence we have that for all index $j$, $J_j=(\inf J_j,\,\sup J_j)$ and $\inf J_j,\,\sup J_j \in \mathcal{C}_k$. Recall that in dimension one, the carré du champ operator takes the form $\Gamma(f)(x)=a(x)f'(x)^2$ with a function $a$ positive in the interior of $M$. So $\mathcal{C}_k$ corresponds to the classical notion of critical points in the interior of $M$.

We now split off the integration by parts formula as follows:

$$\int_M \lambda_k(\nu)f_k\, g - \Gamma(f_k,g)\,d\nu = \sum_j \int_{J_j} \lambda_k(\nu)f_k\, g - \Gamma(f_k,g)\,d\nu $$

Taking $g\in H^1(\nu)$ of the form $g=\phi\circ f_k$ with $\phi : I_k\rightarrow\mathbb{R}$, the above expression becomes:

$$\sum_j \int_{J_j} \lambda_k(\nu)f_k\, \phi(f_k) - \Gamma(f_k)\phi'(f_k)\,d\nu $$

On each $J_j$, $f_k$ is injective since its derivative has constant sign by construction. Hence one can define $h_j : I_k^j \rightarrow \mathbb{R}_+$ by
\begin{equation}\label{defdehj}
h_k^j(t):= \Gamma(f_k)((f_k)_{|J_j}^{-1}(t)) 
\end{equation} 
where $I_k^j:=f_k(J_j)$. So we get that on each $J_j$, $\Gamma(f_k)=h_j\circ f_k$. Therefore it is possible to take the pushforward by $f_k$ on each integral on $J_j$, and the integration by parts formula is transformed into

$$\sum_j \int_{I_k^j} \lambda_k(\nu)t\, \phi(t) - h_j(t)\phi'(t)\,d\nu_j $$
where $\nu_j:= (f_k)_{|J_j}^{\#}(\nu)$ is the pushforward of $\nu$ restricted to $J_j$ by $f_k$. Note that $\nu_j$ is not necessarily a probability distribution (it has a total mass $\nu(J_j)$).

Using the reasoning above, Corollary $\ref{meilleureeqippapproxavecf2}$ can be pushforwarded by $f_k$ and turns into

\begin{equation}\label{ipppushforward}
\left|\sum_j \int_{I_k^j} \lambda_k(\nu)t\, \phi(t) - h_j(t)\phi'(t)\,d\nu_j(t) \right| \leq \left[\sqrt{\left|\lambda_k(\mu)-\lambda_k(\nu)\right|} + \sum_{i=1}^{k-1}C_i\, d(f_k,Sp_i(\nu)^\perp) \right]\sqrt{\sum_j \int_{I_k^j}h_j\phi'^2\,d\nu_j}
\end{equation} where $h$ is defined in Formula $(\ref{defdehj})$ and $\nu_j:= (f_k)_{|J_j}^{\#}(\nu)$ is the pushforward of $\nu$ restricted to $J_j$ by $f_k$. 

Let us emphasize an important feature: while the sets $J_j$ are pairwise disjoint (by construction), their images $I_k^j$ are not disjoint (because $f_k$ is not necessarily injective). Hence the distributions $\nu_j$ do not have disjoint supports in general.

\subsection{Implementing Stein's method}\label{implementsteinhigh}
Inequality $\eqref{ipppushforward}$ will allow us to implement Stein's method. The difference with the usual method is that we will now implement Stein's method on each subinterval $(I_k^j)_j$, and not globally on $I_k$.

Since the Stein equation only depends on the target distribution $\mu$, the quantity $\lambda_k(\nu)$ must not appear anymore on the left hand side of Inequation $(\ref{ipppushforward})$. That is why we begin with writing

$$\int_M \lambda_k(\mu)f_k\,g - \Gamma(g,f_k)\,d\nu   = \int_M \lambda_k(\nu)f_k\,g - \Gamma(g,f_k)\,d\nu  + \left( \lambda_k(\mu)-\lambda_k(\nu)\right)\int_{M} f_k\,g\,d\nu$$

But using the Cauchy-Schwarz inequality, the normalization condition $(\ref{normalisation})$ and the Poincaré inequality $(\ref{poicareimproved})$ with $k=1$, we have
$$\left|\int_{M} f_k\,g\,d\nu \right| \leq \sqrt{\int_M f_k^2 d\nu \int_M g^2d\nu} \leq \frac{1}{\sqrt{\lambda_1(\nu)}}\sqrt{\int_M \Gamma(g)\,d\nu}$$

Hence we have
\begin{equation}\label{steinpret}
\left|\int \lambda_k(\mu)f_k\,g - \Gamma(g,f_k)\,d\nu \right| \leq  \left[\sqrt{\left|\lambda_k(\mu)-\lambda_k(\nu)\right|} +\frac{ \left|\lambda_k(\mu)-\lambda_k(\nu)\right|}{\sqrt{\lambda_1(\nu)}} + \sum_{i=1}^{k-1}C_i d(f_k,Sp_i(\nu)^\perp)\right]\sqrt{\int \Gamma(g)\,d\nu}
\end{equation}

Hence for each $j$, $\phi\mapsto \lambda_k(\mu)t\, \phi (t) - h_j(t)\,\phi'(t)$ is a good candidate to be a Stein operator on $I_k^j$ for the probability distribution $$\mu_j^*:=\frac{1}{\mu(J_j)}{(f_k)_{|J_j}^{\#}(\mu)}.$$ Let us point out that $\mu_j^*$ corresponds to the pushforward by the eigenfunction $f_k$ of the probability distribution $\mu$ restricted to $J_j$.
Our strategy is to implement Stein's method on each $I_k^j$ and then use the approximate integration by parts formula ($\ref{steinpret})$ to deduce a more global result on $I_k$.

\subsubsection{Stein's method on $I_k^j=(a_j,b_j)$}\label{steindecoupe}

On $I_k^j$, the probability measure $\mu_j^*$ is invariant with respect to the diffusion process $$L_j(\psi)(t)= h_j(t)\psi''(t)-\lambda_k(\mu)\,t\,\psi'(t),\quad \forall \psi\in \mathcal{C}^2(I_k^j). $$
As one can see with a classical integration by parts argument, $\mu_j^*$ has therefore the following density with respect to the Lebesgue measure on $I_k^j$: $$d\mu_j^*(t)= \frac{1}{Z_j\, h_j(t)}\exp\left(-\lambda_k(\mu)\int_{a_j}^t \frac{u}{h_j(u)}\,du \right)dt, $$ where $Z_j$ is a normalization constant. Indeed, let $\psi\in \mathcal{C}^2(I_k^j)$ be compactly supported. Then
\begin{align*} & \int_{I_k^j} h_j(t)\psi''(t)\left(\frac{1}{h_j(t)}\exp\left(-\lambda_k(\mu)\int_{a_j}^t \frac{u}{h_j(u)}\,du \right)\right) dt \\
&= \left[ \psi'(t)\exp\left(-\lambda_k(\mu)\int_{a_j}^t \frac{u}{h_j(u)}\,du \right)\right]_{a_j}^{b_j} + \int_{I_k^j} \lambda_k(\mu)\,t\, \psi'(t)\left(\frac{1}{h_j(t)}\exp\left(-\lambda_k(\mu)\int_{a_j}^t \frac{u}{h_j(u)}\,du \right)\right) dt \\
& = \int_{I_k^j} \lambda_k(\mu)\,t\, \psi'(t)\left(\frac{1}{h_j(t)}\exp\left(-\lambda_k(\mu)\int_{a_j}^t \frac{u}{h_j(u)}\,du \right)\right) dt,
\end{align*}
We will then use $h_j(t)\psi'(t)-\lambda_k(\mu)\,t\,\psi(t)$ as a Stein operator for $\mu_j^*$ and the uniqueness of the invariant probability measure for the diffusion process with generator $L(\psi)(t)=h_j(t)\psi''(t)-\lambda_k(\mu)\,t\,\psi'(t)$ allows us to conclude.

Let $g_j: I_k^j\rightarrow \mathbb{R}$ be $1$-Lipschitz, and $\psi_j: I_k^j\rightarrow \mathbb{R}$ given by
\begin{equation}\label{steinsol}
\psi_j(t):=\exp\left( \lambda_k \int_{a_j}^{t}\frac{u\,du}{h_j(u)}\right)\int_{a_j}^{t}\left(g_j(y)-\mu_j^* (g_j)\right) \frac{1}{ h_j(y)}\exp\left( -\lambda_k \int_{a_j}^{y}\frac{u\,du}{h_j(u)}\right) dy. 
\end{equation} 
Then one can easily verify that $\psi_j$ is solution of the Stein equation on $I_k^j$: $$h_j\psi' -\lambda_k(\mu)\,t\,\psi = g_j-\int_{I_k^j} g_jd\mu_j^*. $$
The following estimate holds:
\begin{prop}$\cite[\mathrm{Proposition}\,18]{js}$\label{gammasteinbound}
Let $g_j:I_k^j\rightarrow \mathbb{R}$ be in $ C^1(I_k^j)\cap L^1(\mu_j^*)$, and let $\psi_j$ the associated solution $\eqref{steinsol}$. Then
$$||\sqrt{h_j}\psi_j'||_\infty \leq C_{h_j}\, ||g_j'||_\infty, $$
where 

\begin{align*}
C_{h_j} :=&\, \underset{t \in I_j}{\sup} \left[ \left| 1-Z_j (1-q_j(t))\lambda_k(\mu) \,t\exp\left( \lambda_k(\mu) \int_{a_j}^{t}\frac{u\,du}{h_j(u)}\right)\right|\frac{1}{\sqrt{h_j(t)}}\int_{a_j}^t q_j(y)\,dy \right. \\
+& \left. \left|1+Z_j q_j(x)\lambda_k(\mu)\,t\exp\left( \lambda_k(\mu) \int_{a_j}^{t}\frac{u\,du}{h_j(u)}\right)\right|\frac{1}{\sqrt{h_j(t)}}\int_{t}^{b_j} (1-q_j(y))\,dy \right] 
\end{align*} and $q_j$ is the cumulative distribution function of $\mu_j^*$.
\end{prop}

The issue of finiteness of the constant $C_{h_j}$ is adressed in the following variant of \cite[Proposition 25]{js}.

\begin{prop}\label{resumechfini}
Assume that one of the two following conditions is verified at $a_j$: 
\begin{itemize}
\item either $a_j=-\infty$ and $c_1 |t|^{2\alpha-2}\leq h_j(t)\leq c_2 |t|^\alpha$ for $t\rightarrow -\infty$ with $\alpha\leq 2 $ and $c_1,\,c_2>0$,
\item or $a_j>-\infty$ and $c_1 (t-a_j)^2\leq h_j(t)\leq c_2(t-a_j)$ for $t\rightarrow a_j^+$ with $c_1,\,c_2>0$,
\end{itemize}
and one of these two conditions is satisfied at $b_j$:
\begin{itemize}
\item either $b_j=+\infty$ and $c_1 t^{2\alpha-2}\leq h_j(t)\leq c_2 t^\alpha$ for $t\rightarrow +\infty$ with $\alpha\leq 2 $ and $c_1,\,c_2>0$,
\item or $b_j<+\infty$ and $c_1 (b_j-t)^2\leq h_j(t)\leq c_2(b_j-t)$ for $t\rightarrow b^-$ with $c_1,\,c_2>0$.
\end{itemize}
Then the constant $C_{h_j}$ defined in Proposition $\ref{gammasteinbound}$ is finite.
\end{prop}

In order to ensure the finiteness of these constants $C_{h_j}$, we are lead to make the following assumption.

\begin{assump}\label{hjcommeilfaut}
All $h_j$ satisfy the requirements of Proposition $\ref{resumechfini}$.
\end{assump}

\subsubsection{Proof of Theorem $\ref{stabund}$}

For any $j$, let $\psi_j$ be the Stein solution given by ($\ref{steinsol}$) and define $$\phi(x):=\sum_j \textbf{1}_{J_j}(x)\psi_j(f_k(x)). $$
On the one hand,
\begin{align*}
\int_M \lambda_k(\mu)f_k\phi - \Gamma(\phi,f_k)\,d\nu & = \sum_j \int_{J_j}\lambda_k(\mu)f_k\psi_j(f_k)-\Gamma(f_k,\psi_j(f_k))\,d\nu \\
& = \sum_j \int_{I_k^j} \lambda_k(\mu)t\, \psi_j(t) - h_j(t)\psi_j'(t)\,d\nu_j(t), 
\end{align*}
and by construction, 
\begin{align*}
\sum_j \int_{I_k^j} \lambda_k(\mu)t\, \psi_j(t) - h_j(t)\psi_j'(t)\,d\nu_j(t) & = \sum_j \int_{I_k^j} g_j(t)-\mu_j^*(g_j)\,d\nu_j(t)\\
& = \sum_j \nu_j(g_j)-\nu_j(I_k^j)\mu_j^*(g_j)\\
& = \sum_j \nu_j(I_k^j)\left(\nu_j^*(g_j)-\mu_j^*(g_j)\right), 
\end{align*} 
where $$\nu_j^*:=\frac{1}{\nu_j(I_k^j)}\nu_j $$ is now a probability distribution on $I_k^j$. This is the pushforward by the eigenfunction $f_k$ of the probability distribution $\nu$ restricted to $J_j$. Let us point out that $\nu_j(I_k^j)=\nu(J_j) $.

On the other hand, 
\begin{align*}
\sum_j \int_{I_k^j}h_j(t)(\psi'_j)^2(t)\,d\nu_j(t) = & \sum_j \int_{J_j}h_j(f_k(x))(\psi'_j(f_k(x)))^2\,d\nu(x)\\
& = \sum_j \int_{J_j} \Gamma(\psi_j(f_k))(x)\,d\nu(x)\\
& =\int_M\Gamma(\phi)(x)\,d\nu(x), 
\end{align*}
hence Proposition $\ref{gammasteinbound}$ gives $$\int_M\Gamma(\phi)(x)\,d\nu(x)\leq \sum_j \nu_j(I_k^j)C_{h_j}^2 \leq \sum_j C_{h_j}^2.  $$
So taking $g=\phi$ in ($\ref{steinpret}$), one gets $$\sup_{(g_j)_j} \left|\sum_j \nu_j(I_k^j)\left(\nu_j^*(g_j)-\mu_j^*(g_j)\right)\right|\leq C \left[\sqrt{\left|\lambda_k(\mu)-\lambda_k(\nu)\right|} +\frac{ \left|\lambda_k(\mu)-\lambda_k(\nu)\right|}{\sqrt{\lambda_1(\nu)}} + \sum_{i=1}^{k-1}C_i d(f_k,Sp_i(\nu)^\perp)\right]$$
where the supremum runs over all $(\#\mathcal{C}_k)$-uple of functions $(g_j)_j$ with for all $j$, $g_j:I_k^j\rightarrow\mathbb{R}$ being $1$-Lipschitz, and $C:=\sum_j C_{h_j}^2$.
Finally $$\sup_{(g_j)_j} \left|\sum_j \nu_j^*(I_k^j)\left(\nu_j^*(g_j)-\mu_j^*(g_j)\right)\right| = \sum_j \nu_j(I_k^j)\, W_1(\nu_j^*,\mu_j^*). $$
Indeed, the inequality "$\leq$" easily follows from the triangle inequality.  To see the other direction "$\geq$", let $\varepsilon>0$ small enough and for all $j$, pick a $1$-Lipschitz function $g_j$ such that $\mu_j^*(g_j)-\nu_j^*(g_j)\geq W_1(\nu_j^*,\mu_j^*) -\varepsilon $. Then 
\begin{align*}
\left|\sum_j \nu_j(I_k^j)\left(\mu_j^*(g_j)-\nu_j^*(g_j)\right)\right| & = \sum_j \nu_j(I_k^j)\left(\mu_j^*(g_j)-\nu_j^*(g_j)\right)\\
& \geq \sum_j \nu_j(I_k^j)\left(W_1(\nu_j^*,\mu_j^*) -\varepsilon\right) \\
& = \sum_j \nu_j(I_k^j)\,W_1(\nu_j^*,\mu_j^*) -\varepsilon.  
\end{align*} Letting $\varepsilon$ go to zero concludes the proof.

\section{Application to the Gaussian distribution}\label{highergauss}

In this section, we consider the case of the one dimensional Ornstein-Uhlenbeck operator, where $M=\mathbb{R}$, $Lf=f''-xf'$ and $\mu=\gamma:=\mathcal{N}(0,1)$  is the equilibrium distribution. The carré du champ operator is $\Gamma(f,g)=f'\,g'$. The eigenvalues are all integers: $\lambda_k(\gamma)=k$, with multiplicity $1$, and the associated normalized eigenfunctions are the Hermite polynomials $H_k$, $k\geq 1$ given by $$H_k=\frac{1}{\sqrt{n!}}P_k $$ where $P_0(x)=1$, $P_1(x)=x$ and $$P_{n+1}(x)= xH_n(x)-nP_{n-1}(x). $$ Note that the polynomial $H_0:=P_0=1$ (which is not centered) corresponds to the zero-th eigenvalue $\lambda_0=0$, so only the $H_k$ with $k\geq 1$ are relevant. 

\subsection{The second eigenvalue}

The case of the second eigenvalue $\lambda_2=2$, and $f_2(x)=\frac{1}{\sqrt{2}}\left(x^2-1\right)$, is quite specific. Indeed, the pushforward measure $f_2^{\#}(\gamma)=\frac{1}{\sqrt{2}}\left(\chi_2-1\right) $ corresponds to a translation of the Chi-$2$ distribution. Let us mention that Chi-$2$ approximation have been investigated through the tools of the Stein-Malliavin method in \cite{secondwienerchaoschi2, steinpourchi2}. 

Let us apply our result. One could directly apply Theorem $\ref{stabund}$: zero is the only critical point, so we can inverse $f_2$ on the connected components $\mathbb{R}_-$ and $\mathbb{R}_+$ and deduce factorizations for $\Gamma(f_2)=2x^2$ on each of these connected components, allowing to implement Stein's method after taking the pushforward by $f_2$. One can then see that Assumption $\ref{hjcommeilfaut}$ is satisfied, and deduce the split stability estimate of Theorem $\ref{stabund}$.

However, in this case, though $f_2$ is not injective, $\Gamma(f_2)$ can anyway be globally factorized on $\mathbb{R}$. Actually, $\Gamma(f_2)(x) = 2x^2 = 2\sqrt{2}\left(\frac{1}{\sqrt{2}}\left(x^2-1\right) +\frac{1}{\sqrt{2}}\right) = 2\sqrt{2}\left( f_2(x)+\frac{1}{\sqrt{2}}\right)$. Hence $h:[-\frac{1}{\sqrt{2}},\infty)\rightarrow \mathbb{R}_+$ given by $h(t)=2\sqrt{2}(t+\frac{1}{\sqrt{2}})$ globaly factorizes $\Gamma(f_2)$. So from the approximate integration by parts formula ($\ref{ippapproch}$), instead of using the method of Section $\ref{splitpushforward}$, we can proceed as for the first eigenfunction and taking the gloal pushforward by $f_2$ on all $\mathbb{R}$. The following Stein operator is obtained:
$$ 2t\,f(t)-2\sqrt{2}\left(t+\frac{1}{\sqrt{2}}\right)f'(t) \quad \mathrm{on}\quad (-\frac{1}{\sqrt{2}},+\infty).$$
Let us underline that this generator of the $f_2^{\#}(\gamma)=\frac{1}{\sqrt{2}}\left(\chi_2-1\right) $ distribution corresponds to the one used in \cite{steinpourchi2}.
One can see that this $h$ satisfies the conditions required by Proposition $\ref{resumechfini}$. Indeed $h$ is an affine function, so the vanishing rate condition at $a=-\frac{1}{\sqrt{2}}$ is obvously satisfied, and $\alpha=1$ is a suitable choice for the growth condition at $b=+\infty$.
Moreover, the normalization conditions ($\ref{normalisation}$) are here reduced to the following two normalizations on the moments of order $2$ and $4$: $\nu$ is asked to have the same moments of order $2$ and $4$ than the standard normal distribution.

So the following is proven:
\begin{thm}
For all measure $\nu$ on $\mathbb{R}$ normalized as $$\int x^2d\nu=1,\quad \mathrm{and}\quad \int x^4\,d\nu = 3, $$ and satisfying an improved Poincaré inequality with sharp constant $\frac{1}{\lambda_2(\nu)}$, it holds for some finite positive constant $C>0$:
$$W_1\left(\frac{1}{\sqrt{2}}\left(\chi_2 -1\right),\nu^*\right) \leq C \left[\sqrt{\left|2-\lambda_2(\nu)\right|} +\frac{\left|2-\lambda_2(\nu)\right|}{\sqrt{\lambda_1(\nu)}} + C_\nu\, d\left(\frac{1}{\sqrt{2}}\left(x^2-1\right),Sp_1(\nu)^\perp\right) \right]$$
where $\chi_2$ is the $\chi_2$-distribution on $\mathbb{R}_+$, $\nu^*$ is the pushforward of $\nu$ by $f_2=\frac{1}{\sqrt{2}}\left(x^2-1\right)$, $W_1$ is the $1$-Wasserstein distance, the constant $C_\nu$ is given by $C_\nu=\sqrt{\lambda_2(\nu)-\lambda_1(\nu)} +\frac{\lambda_2(\nu)-\lambda_1(\nu)}{\sqrt{\lambda_1(\nu)}}$, and $d\left(\frac{1}{\sqrt{2}}\left(x^2-1\right),Sp_1(\nu)^\perp\right)$ quantifies the orthogonality error between $\frac{1}{\sqrt{2}}\left(x^2-1\right)$ and the first eigenspace of $\nu$.
\end{thm}

\subsection{The $k$-th eigenvalue, $k\geq 3$}

As soon as $k\geq 3$, the global factorization $\Gamma(f_k)=h\circ f_k$ does not hold anymore. We are therefore led to use the method explained in Section $\ref{splitpushforward}$. Since 
\begin{equation}\label{derivehermite}
\forall n\geq 0,\,\,H_{n+1}'=\sqrt{n+1}\, H_n\,,
\end{equation} one gets that the critical points are $\mathcal{C}_k= \{ H_{k-1} = 0\} $. So the connected components of $\mathbb{R}\setminus\mathcal{C}_k$ are the $k$ nodal sets of $H_{k-1}$. The eigenfunction $H_k$ is injective on each of the connected component, $J_j$, so $\Gamma(H_k)$ factorizes as $\Gamma(H_k)=h_j\circ H_k$. Since $\Gamma(H_k)=(H_k')^2=k\,H_{k-1}^2$ we get $$h_j = k\,H_{k-1}^2\circ (H_k)_{|J_j}^{-1}. $$  Let us show that these functions $h_j$ satisfy Assumption $\ref{hjcommeilfaut}$. At an infinite boundary, since $H_k(x) \sim \frac{x^k}{\sqrt{k!}}$, we get $ h_j(t)\sim_C t^{2\frac{k-1}{k}}$, where $f\sim_C g$ means that $\frac{f}{g}$ tends to a constant. Therefore $\alpha=2\frac{k-1}{k}$ is a suitable choice in Proposition $\ref{resumechfini}$. Let us now treat the case of a finite boundary. We start by showing the following fact for Hermite polynomials:
\begin{fact}
For all $n\geq 0$, if $x_0\in \{\inf J_j,\sup J_j\}$ and $y_0:=H_{n+1}(x_0)$, then for $y\in I_k^j $ close enough to $y_0$, one has $$\left|(H_{n+1})_{|J_j}^{-1}(y)-y_0 \right|\leq c \sqrt{y-y_0}, $$ for some $c>0$.
\end{fact}
\begin{dem}
Since $\inf J_j$ and $\sup J_j$ are critical points of $H_{n+1}$, this fact is equivalent to the fact that $H_{n+1}$ is quadratic at the neighborhood of all of its critical points, i.e. there is some non zero $c$ such that $H_{n+1}(x)-H_{n+1}(x_0) = c(x-x_0)^2 + o(x-x_0)^2$ for all critical points $x_0$. To show this we are reduced to check that $H_{n+1}''(x_0)\neq 0$. But using again formula ($\ref{derivehermite}$), this is true because two consecutive Hermite polynomials never have a common root.
\end{dem}

Then, since $\Gamma(H_k)=k\,H_{k-1}^2$ and since the roots of Hermite polynomials have only multiplicity one, we deduce that at all finite boundaries of $I_k^j$, the function $h_j$ vanishes at a linear rate.
Finally the requirement of Proposition $\ref{resumechfini}$ is satisfied, so that Assumption $\ref{hjcommeilfaut}$ is satisfied.
Therefore in this case the normalization conditions ($\ref{normalisation}$) take the form $$\int H_k\,d\nu=0, \quad \int H_k^2d\nu=1\,\,\mathrm{and}\,\, \int H_{k-1}^2d\nu=1. $$ Unlike what happened for the first and the second eigenfunctions, these three conditions can not be reduced to two. For example, for $k=3$, they are: $$\int (x^3-3x)\,d\nu=0,\quad \int (x^6-6x^4+9x^2)\,d\nu=6\,\, \mathrm{and}\,\, \int (x^4-2x^2)\,d\nu=1. $$

The requirements of Theorem $\ref{stabund}$ being satisfied, we can then apply it and get the following stability result for higher order eigenvalues of the one dimensional normal distribution.

\begin{thm}
Let $k\geq 3$ and $H_k$ $($resp. $H_{k-1})$ be the $k$-th $($resp. $(k-1)$-th$)$ Hermite polynomial.
Then for all probability measures $\nu$ on $M$ normalized as $$\int H_k\,d\nu=0, \quad \int H_k^2d\nu=1\,\,\mathrm{and}\,\, \int H_{k-1}^2d\nu=1, $$
and satisfying the improved Poincaré inequalities $\eqref{poicareimproved}$, it holds for some finite constant $C>0$:
$$ \sum_j \nu(J_j)\, W_1(\nu_j^*,\gamma_j^*) \leq C \left[\sqrt{\left|k-\lambda_k(\nu)\right|} +\frac{ \left|k-\lambda_k(\nu)\right|}{\sqrt{\lambda_1(\nu)}} + \sum_{i=1}^{k-1}C_i\, d(H_k,Sp_i(\nu)^\perp)\right]  $$
where $(J_j)_j$ are the connected components of the complementary of critical points of $H_k$, $\nu_j^*$ (resp. $\gamma_j^*$) is the pushforward of $\nu$ (resp. $\gamma$) restricted to $H_k(J_j)$, constants $C_i$ are given by $$C_i=\sqrt{\lambda_k(\nu)-\lambda_i(\nu)}+\frac{\lambda_k(\nu)-\lambda_i(\nu)}{\sqrt{\lambda_i(\nu)}},$$ and $d(H_k,Sp_i(\nu)^\perp)$ is defined in Remark $\ref{orthoerror}$ and quantifies the orthogonality error between $H_k$ and eigenspaces of lower orders of $\nu$.
\end{thm}

\section{Application to Gamma distributions}\label{higherlaguerre}

In this section, we consider the case of the Laguerre operator, where $M=\mathbb{R_+}$, $Lf=xf''+(s-\frac{x}{\theta})f'$ and $\mu$ is the $\Gamma(s,\theta)$ distribution on $\mathbb{R}_+$ given by the density $d\mu(x)=\frac{x^{s-1}e^{-\frac{x}{\theta}}}{\Gamma(s)\theta^s}\textbf{1}_{\mathbb{R}_+}$, where $\Gamma$ denotes the Euler $\Gamma$ function. The carré du champ operator is $\Gamma(f,g)=x\,f'\,g'$. The eigenvalues are $\lambda_k(\mu)=\frac{k}{\theta}$, with multiplicity $1$, and the associated normalized eigenfunctions are the generalized normalized Laguerre polynomials $L_{k,s}$ given by  
\begin{equation}\label{deflaguerre}
L_{k,s}(x)=\sqrt{\frac{k!\,\Gamma(s)}{\Gamma(k+s)}}\,l_{k,s}\left(\frac{x}{\theta}\right)
\end{equation}
where $l_{0,s}(x)=1$, $l_{1,s}(x)=s-x$, and $$\forall n\geq1,\quad l_{n+1,s}(x)= \left( 2+\frac{s-2-x}{n+1}\right)l_{n,s}(x) - \left(1+\frac{s-2}{n+1}\right)l_{n-1,s}(x). $$

Moreover, 
\begin{equation}\label{derivedeslaguerres}
L_{n,s}'(x)=\frac{-1}{\theta}\sqrt{\frac{n}{s}}\,L_{n-1,s+1}(x). 
\end{equation}
So the critical points $\mathcal{C}_k=\{x\in\mathbb{R}_+\,|\,\Gamma(L_{k,s})(x)=0\}$ of $L_{k,s}$ are the zeros of $L_{k-1,s+1}$ and $0$. This means that there are $k$ connected components $J_j$ of $\mathbb{R}_+\setminus\mathcal{C}_k$, which are the nodal sets of $L_{k-1,s+1}$.

The eigenfunction $L_{k,s}$ is injective on each connected component $J_j$, so $\Gamma(L_{k,s})$ factorizes as $\Gamma(L_{k,s})=h_j\circ L_{k,s}$. Since $\Gamma(L_{k,s})(x)=x\,(L_{k,s}'(x))^2= \frac{kx}{s\theta^2}(L_{n-1,s+1}(x))^2 $ we get 
\begin{equation}\label{hjlaguerre}
h_j(t) = \frac{k}{s\theta^2}(L_{k,s})_{|J_j}^{-1}(t)(L_{n-1,s+1}\circ (L_{k,s})_{|J_j}^{-1}(t))^2.
\end{equation}
Let us show that these functions $h_j$ satisfy Assumption $\ref{hjcommeilfaut}$. At $+\infty	$, since $L_{k,s}(x)\sim \frac{1}{\theta^k}\sqrt{\frac{n!\,\Gamma(s)}{\Gamma(n+s)}}\,x^k$, we get $$h_j(t)\sim_C t^{\frac{1}{k}+\frac{2(k-1)}{k}} = t^{2-\frac{1}{k}}. $$ Therefore $\alpha=2-\frac{1}{k}$ is a suitable choice for $\alpha$ in Proposition $\ref{resumechfini}$. Let us treat now the case of a finite boundary. In the same way as for Hermite polynomials, we can see the following.
\begin{fact}
Let $n\geq 0$, let $x_0\in \{\inf J_j,\sup J_j\}$ such that $L_{n,s}'(x_0)=0$, and set $y_0:=L_{n,s}(x_0)$. Then for $y\in I_k^j $ close enough to $y_0$, one has $$\left|(L_{n,s})_{|J_j}^{-1}(y)-y_0 \right|\leq c\sqrt{y-y_0}, $$ for some $c>0$.
\end{fact}
\begin{dem}
This fact is equivalent to the fact that $L_{n,s}$ is quadratic at the neighborhood of all of its critical points, i.e. there is some non zero $c$ such that $L_{n,s}(x)-L_{n,s}(x_0) = c(x-x_0)^2 + o(x-x_0)^2$ for all critical points $x_0$. To show this we are reduced to check that $L_{n,s}''(x_0)\neq 0$. But using Formula ($\ref{derivedeslaguerres}$), this  would imply that $x_0$ is a root of $L_{n-1,s+1}$ with multiplicity at least two. However, all Laguerre polynomials only have roots with multiplicity one. So the fact is proven.
\end{dem}

Using the above fact and Formulas ($\ref{derivedeslaguerres}$) and ($\ref{hjlaguerre}$), we get that at avery finite boundary of $I_k^j$, the function $h_j$ vanishes at a linear rate.
Hence Assumption $\ref{hjcommeilfaut}$ is satisfied.

Moreover in this case the normalization conditions ($\ref{normalisation}$) takes the form $$\int L_{k,s}\,d\nu=0, \quad \int L_{k,s}^2d\nu=1\,\,\mathrm{and}\,\, \int x\,(L_{k-1,s+1}(x))^2d\nu= s\theta, $$ where $\Gamma$ is the Euler Gamma function, and $!$ denotes the factorial. For $k=1$ these three conditions reduced to only two, but as soon as $k\geq 2$ it is not the case anymore. For example, if $k=2$, and $s=\theta=1$ (in that case $\mu$ is the exponential distribution), then these three conditions are: $$\int \frac{x^2}{2} -2x\,d\mu =1, \quad \int \frac{x^4}{4}-2x^3+5x^2-4x\,d\mu=0,\quad \mathrm{and}\,\, \int x^3-4x^2+4x\,d\mu = 2.$$

The requirements of Theorem $\ref{stabund}$ being satisfied, we can then apply it and get the following stability result for higher order eigenvalues of the Gamma distributions $\Gamma(s,\theta)$ on $\mathbb{R}_+$.

\begin{thm}\label{stablaguerrehigh}
Let $k\geq 1$, $s>0$, $\theta>0$, and $L_{k,s}$ be the Laguerre polynomials defined in $\eqref{deflaguerre}$.
Then for all probability measures $\nu$ on $\mathbb{R}_+$ normalized with $$\int L_{k,s}\,d\nu=0, \quad \int L_{k,s}^2d\nu=1\,\,\mathrm{and}\,\, \int x\,(L_{k-1,s+1}(x))^2d\nu= s\theta, $$ where $\Gamma$ is the Euler Gamma function, and satisfying the improved Poincaré inequalities $\eqref{poicareimproved}$, it holds for some finite constant $C>0$:

$$ \sum_j \nu_j(I_k^j)\, W_1(\nu_j^*,\mu_j^*) \leq C \left[\sqrt{\left|\frac{k}{\theta}-\lambda_k(\nu)\right|} +\frac{ \left|\frac{k}{\theta}-\lambda_k(\nu)\right|}{\sqrt{\lambda_1(\nu)}} + \sum_{i=1}^{k-1}C_i\, d(L_{k,s},Sp_i(\nu)^\perp)\right]  $$
where $(I_k^j)_j$ are the images by $L_{k,s}$ of the connected components of the complementary of its critical points, $\nu_j^*$ (resp. $\mu_j^*$) is the pushforward of $\nu$ (resp. $\mu$) restricted to $I_k^j$, constants $C_i$ are given by $$C_i=\sqrt{\lambda_k(\nu)-\lambda_i(\nu)}+\frac{\lambda_k(\nu)-\lambda_i(\nu)}{\sqrt{\lambda_i(\nu)}},$$ and $d(L_{k,s},Sp_i(\nu)^\perp)$ is defined in Remark $\ref{orthoerror}$.
\end{thm}

\section{Application to $\beta\left(\frac{N}{2},\frac{N}{2}\right)$ distributions}\label{highjacobi}

In this section, we consider the case of the Jacobi operator, where $M=[-1,1]$,
$$Lf(x) = (1-x^2)f''(x)-Nxf'(x), $$
and $\mu$ is the $\beta\left(\frac{N}{2},\frac{N}{2}\right)$ distribution on $[-1,1]$ given by the density $$d\mu(x)=\frac{1}{Z}(1-x^2)^{\frac{N}{2}-1}dx,$$ where $$Z=2^{2-2N}\pi\frac{\Gamma(d-1)}{\left(\frac{d-1}{2}\right)\left(\Gamma(\frac{d-1}{2}\right)^2}$$ is the normalization constant and $\Gamma$ denotes the Euler function. The carré du champ operator is given by $\Gamma(f,g)(x)=(1-x^2)^{\frac{N}{2}-1}f'(x)g'(x)$. The eigenvalues are $\lambda_k(\mu)=k(k+N-1)$, with multiplicity one, and the associated normalized eigenfunctions are the normalized Gegenbauer polynomials (see \cite{szegopolynomials}) given by
\begin{equation}\label{defgegenbauer}
G_{N,k}(x) = \begin{pmatrix}
k+N-2 \\
k
\end{pmatrix}^{-1}\frac{2k+N-1}{N-1}\, P_{N,k}(x),
\end{equation}
where $P_{N,0}(x)=1$, $P_{N,1}(x)=(N-1)x$, and
$$P_{N,k}(x)= \frac{2x}{k}\left(k+\frac{N-3}{2}\right)P_{N,k-1}(x) - \frac{1}{k}(k+N-3)P_{N,k-2}(x). $$
The Jacobi operator corresponds to the Laplace-Beltrami operator on the sphere $\mathbb{S}^N$ projected on one coordinate and normalized to stay in $[-1,1]$. The Gegenbauer polynomials are particular case of Jacobi polynomials, when the two parameters of Jacobi polynomials are equals.

\subsection{The second eigenvalue}

Similarly to the case of the Normal distribution, the global factorization condition of the carré du champ is satisfied for $k=1$ and $k=2$. Indeed, the second Gegenbauer polynomial is $$G_{N,2}(x)= \frac{1}{2}\begin{pmatrix}
N\\
2
\end{pmatrix}^{-1} (N+3)\left((N+1)x^2-1\right),$$ so we can compute
$$\Gamma(G_{N,2})(x) =K^2(1-x^2)x^2 =K^2\left[1-\frac{1}{N+1}\left(\frac{1}{K}G_{N,2}(x)+1\right)\right]\frac{1}{N+1}\left(\frac{1}{K}G_{N,2}(x)+1\right), $$ where $K=\begin{pmatrix}
N\\
2
\end{pmatrix}^{-1} (N+3)(N+1) $. Hence $\Gamma(G_{N,2})(x) = h(G_{N,2}(x))$, with $$h(t):= \frac{K}{N+1}\left(\frac{N}{N+1}-\frac{t}{K(N+1)}\right)\left(\frac{t}{K}+1\right). $$ As in the Gaussian case, this is due to the fact that the only critical point is zero, and the carré du champ is symetric. This $h$ satisfies the vanishing rate requirements (it vanish at linear speed), so we have the following stability result.
\begin{thm}
For all measure $\nu$ on $\mathbb{R}$ satisfying $$\int x^2d\nu=\frac{1}{N+1},\quad \mathrm{and}\quad \int x^4\,d\nu = \frac{1}{N+1}\left[4\begin{pmatrix}
N\\
2
\end{pmatrix}^2(N+3)^{-2} +1\right] $$ and an improved Poincaré inequality with sharp constant $\frac{1}{\lambda_2(\nu)}$, it holds for some finite positive constant $C>0$ that
$$W_1\left(L\left((N+1)\beta_N^2 -1\right),\nu^*\right) $$ 
$$\leq C \left[\sqrt{\left|2(N+1)-\lambda_2(\nu)\right|} +\frac{\left|2(N+1)-\lambda_2(\nu)\right|}{\sqrt{\lambda_1(\nu)}} + C_\nu\, d\left(L\left((N+1)\beta_N^2 -1\right),Sp_1(\nu)^\perp\right) \right]$$
where $\beta_N$ is the $\beta\left(\frac{N}{2},\frac{N}{2}\right)$-distribution on $[-1,1]$, $\nu^*$ is the pushforward of $\nu$ by $f_2=L\left((N+1)x^2-1\right)$, the constant $L$ is given by $L=\begin{pmatrix}
N\\
2
\end{pmatrix}^{-1} (N+3)$, the constant $C_\nu$ is given by $C_\nu=\sqrt{\lambda_2(\nu)-\lambda_1(\nu)} +\frac{\lambda_2(\nu)-\lambda_1(\nu)}{\sqrt{\lambda_1(\nu)}}$, and $d\left(L\left((N+1)\beta_N^2 -1\right),Sp_1(\nu)^\perp\right) $ quantifies the orthogonality error between $f_2$ and the first eigenspace of $\nu$.
\end{thm}

\subsection{The $k$-th eigenvalue, $k\geq 3$}

As soon as $k\geq 3$, the global factorization does not hold anymore. So we apply the general method presented in Section $\ref{splitpushforward}$. We have
\begin{equation}\label{derivegegenbauer}
G_{N,k}'(x) = \begin{pmatrix}
k+N-1\\
k-1
\end{pmatrix}
\begin{pmatrix}
k+N-2\\
k
\end{pmatrix}^{-1}
(N+1) \, G_{N+2,k-1}(x).
\end{equation}

So the critical points $\mathcal{C}_k=\{x\in\mathbb{R}_+\,|\,\Gamma(G_{N,k})(x)=0\}$ of $G_{N,k}$ are the zeros of $G_{N+2,k-1}$ and $-1$ and $1$. This means that there are $k$ connected components $J_j$ of $[-1,1]\setminus\mathcal{C}_k$ which are the nodal sets of $G_{N+2,k-1}$. The eigenfunctions $G_{N,k}$ are injective on each of this connected component $J_j$, so $\Gamma(G_{N,k})$ factorizes as $\Gamma(G_{N,k})=h_j\circ G_{N,k}$. Since $\Gamma(G_{N,k})=C(1-x^2)(G_{N+2,k-1}(x))^2$ where $$C=\begin{pmatrix}
k+N-1\\
k-1
\end{pmatrix}^2
\begin{pmatrix}
k+N-2\\
k
\end{pmatrix}^{-2}
(N+1)^2 ,$$ we get
\begin{equation}\label{hjjacobi}
h_j(t) = C\left(1 - (G_{N,k\,|J_j}^{-1}(t))^2 \right) \left(G_{N+2,k-1}\circ G_{N,k\,|J_j}^{-1}(t)\right)^2.
\end{equation}
In order to Theorem $\ref{stabund}$ to apply, we have to verify that these functions $h_j$ satisfy Assumption $\ref{hjcommeilfaut}$. Gegenbauer polynomials do not vanish at $-1$ and $1$, so the rate at which $h_j$ vanishes at boundaries $G_{N,k}(-1)$ and $G_{N,k}(1)$ is linear. We have then to treat the case of critical points in the interior of $[-1,1]$. The reasonning is the same as for Laguerre polynomials: since Gegenbauer polynomials have only roots of multiplicity one, by a Taylor expansion we see that the functions $h_j$ vanishes at boundaries of their domains of definition with linear rate.

Moreover, in this case, the normalization conditions $\eqref{normalisation}$ take the form
$$\int G_{N,k}\,d\mu=0,\quad \int G_{N,k}^2d\mu=1,\quad$$
$$ \mathrm{and}\quad \int  (1-x^2)G_{N+2,k-1}^2d\mu = \begin{pmatrix}
k+N-1\\
k-1
\end{pmatrix}^{-1}
\begin{pmatrix}
k+N-2\\
k
\end{pmatrix}
\frac{k(k+N-1)}{N+1}.  $$

Finaly, Theorem $\ref{stabund}$ can be applied, and one gets the following.
\begin{thm}\label{stabbetahigh}
Let $N>1$, $\mu$ be the $\beta\left(\frac{N}{2},\frac{N}{2}\right)$ distribution on $[-1,1]$, and $G_{N,k}$ be the Gegenbauer polynomials defined in $\eqref{defgegenbauer}$.
Then for all probability measures $\nu$ on $[-1,1]$ normalized such that $$\int G_{N,k}\,d\mu=0,\quad \int G_{N,k}^2d\mu=1,\quad$$
$$ \mathrm{and}\quad \int  (1-x^2)G_{N+2,k-1}^2d\mu = \begin{pmatrix}
k+N-1\\
k-1
\end{pmatrix}^{-1}
\begin{pmatrix}
k+N-2\\
k
\end{pmatrix}
\frac{k(k+N-1)}{N+1},$$ and satisfying the improved Poincaré inequalities $\eqref{poicareimproved}$, it holds for some finite constant $C_\beta>0$:
$$ \sum_j \nu_j(I_k^j)\, W_1(\nu_j^*,\mu_j^*) \leq C_\beta \left[\sqrt{\left|k(k+N-1)-\lambda_k(\nu)\right|} +\frac{ \left|k(k+N-1)-\lambda_k(\nu)\right|}{\sqrt{\lambda_1(\nu)}} + \sum_{i=1}^{k-1}C_i d(L_{k,s},Sp_i(\nu)^\perp)\right]  $$
where $(I_k^j)_j$ are the images by $G_{N,k}$ of the connected components of the complementary of its critical points, $\nu_j^*$ (resp. $\mu_j^*$) is the pushforward of $\nu$ (resp. $\mu$) restricted to $I_k^j$, constants $C_i$ are given by $$C_i=\sqrt{\lambda_k(\nu)-\lambda_i(\nu)}+\frac{\lambda_k(\nu)-\lambda_i(\nu)}{\sqrt{\lambda_i(\nu)}},$$ and $d(G_{N,k},Sp_i(\nu)^\perp)$ is defined in Remark $\ref{orthoerror}$ and quantifies the orthogonality error between $G_{N,k}$ and eigenspaces of lower orders of $\nu$.
\end{thm}

\textbf{Acknowledgments:} {This work was supported by the Labex Cimi of University of Toulouse. I would like to thank both my PhD advisors for this present work, Max Fathi for all the discussions on this wide topic, and Franck Barthe for all his fruitful advice.}

\bibliographystyle{plain}
\bibliography{mabibliographie}

\end{document}